\documentclass[11pt, reqno, oneside]{amsart}
\usepackage{amsmath,amsfonts,amssymb,amsthm}

\usepackage{enumerate}
\usepackage{esint}
\usepackage[pagebackref,colorlinks,citecolor=blue,linkcolor=blue]{hyperref}
\usepackage[dvipsnames]{xcolor}
\usepackage{mathtools}  

\usepackage{mathtools}
\usepackage{setspace}

\usepackage{comment}

\usepackage{geometry}
\numberwithin{equation}{section}

\theoremstyle{plain}
\newtheorem{theorem}[equation]{Theorem}
\newtheorem{corollary}[equation]{Corollary}
\newtheorem{lemma}[equation]{Lemma}
\newtheorem{proposition}[equation]{Proposition}

\theoremstyle{definition}
\newtheorem{definition}[equation]{Definition}

\newtheorem{remark}[equation]{Remark}

\newcommand{\R}{{\mathbb R}}
\newcommand{\N}{{\mathbb N}}

\newcommand{\Om}{\Omega}

\providecommand{\vint}[1]{\mathchoice
	{\mathop{\vrule width 5pt height 3 pt depth -2.5pt
			\kern -9pt \kern 1pt\intop}\nolimits_{\kern -5pt{#1}}}
	{\mathop{\vrule width 5pt height 3 pt depth -2.6pt
			\kern -6pt \intop}\nolimits_{\kern -3pt{#1}}}
	{\mathop{\vrule width 5pt height 3 pt depth -2.6pt
			\kern -6pt \intop}\nolimits_{\kern -3pt{#1}}}
	{\mathop{\vrule width 5pt height 3 pt depth -2.6pt
			\kern -6pt \intop}\nolimits_{\kern -3pt{#1}}}}

\newcommand{\eps}{\varepsilon}
\newcommand{\loc}{\mathrm{loc}}

\newcommand{\BV}{\mathrm{BV}}

\newcommand{\ch}{\text{\raise 1.3pt \hbox{$\chi$}\kern-0.2pt}}

\newcommand{\mres}{\mathbin{\vrule height 2ex depth 2.2pt width
		0.12ex\vrule height -0.3ex depth 2.2pt width .5ex}}

\DeclareMathOperator{\AM}{AM}
\DeclareMathOperator{\capa}{Cap}

\DeclareMathOperator{\dive}{div}
\DeclareMathOperator{\dist}{dist}
\DeclareMathOperator{\diam}{diam}

\DeclareMathOperator{\Var}{Var}

\DeclareMathOperator{\osc}{osc}

\begin{document}
\title[]{Alberti's rank one theorem and quasiconformal mappings
	 in metric measure spaces
}
\author{Panu Lahti}
\address{Panu Lahti,  Academy of Mathematics and Systems Science, Chinese Academy of Sciences,
	Beijing 100190, PR China, {\tt panulahti@amss.ac.cn}}


\begin{abstract}
	We investigate a version of Alberti's rank one theorem in Ahlfors regular metric spaces,
	as well as a connection with quasiconformal mappings.
More precisely, we give a proof of the rank one theorem that partially
follows along the usual steps,
but the most crucial step consists in showing for
$f\in\BV(X;Y)$ that at $\Vert Df\Vert^s$-a.e. $x\in X$, the mapping $f$
``behaves non-quasiconformally''.
\end{abstract}

\date{\today}
\maketitle

\section{Introduction}

Alberti's rank one theorem \cite{Alb} states that for a function of bounded variation
$f\in \BV(\R^n;\R^k)$,
with $n,k\in\N$, $\tfrac{dDf}{d|Df|}(x)$ has rank one for $|D^sf|$-a.e. $x\in\R^n$.
Other proofs and generalizations have been given in \cite{DLe, DPhRi, MaVi}.
It is well known that proving the rank one theorem can be reduced to the case of
functions $f\in \BV(\R^2;\R^2)$; note that 
a nonzero $2\times 2$-matrix having rank one is equivalent to it not having full rank.

Quasiconformal mappings are defined as homeomorphisms $f\in W_{\loc}^{1,n}(\R^n;\R^n)$, $n\ge 2$,
for which $|\nabla f(x)|\le K|\det \nabla f(x)|$ for a.e. $x\in\R^n$ and some constant
$K<\infty$. In particular, $\nabla f$ has full rank a.e.
Heuristically, the rank one theorem thus states that at $|D^sf|$-a.e. $x\in\R^n$,
$f$ behaves in a ``non-quasiconformal'' way.
Quasiconformal mappings can also be defined between
two metric spaces $(X,d)$ and $(Y,d_Y)$. For a mapping $f\colon X\to Y$,
for every $x\in X$ and $r>0$ one defines
\[
L_f(x,r):=\sup\{d_Y(f(y),f(x))\colon d(y,x)\le r\}
\]
and
\[
l_f(x,r):=\inf\{d_Y(f(y),f(x))\colon d(y,x)\ge r\},
\]
and
\[
H_f(x,r):=\frac{L_f(x,r)}{l_f(x,r)};
\]
we interpret this to be $\infty$ if the denominator is zero.
Then one defines
\[
	h_f(x):=\liminf_{r\to 0} H_f(x,r).
\]
A homeomorphism $f\colon X\to Y$ is said to be (metric) \emph{quasiconformal} if there is a number
$1\le H<\infty$ such that $h_f(x)\le H$ for all $x\in X$.
In Euclidean spaces and sufficiently regular metric measure spaces,
this is equivalent with the earlier mentioned ``analytic'' definition,
see  e.g. \cite[Theorem 9.8]{HKST}.

All of this raises the question of whether one can formulate a precise connection
between quasiconformal mappings and Alberti's rank one theorem, and whether a formulation of the latter
exists in metric measure spaces.
In this paper we will
give a proof of the rank one theorem that partially follows along the usual steps, but the
most crucial step is proved in very general metric measure spaces using
a notion of quasiconformality.
To this end, for a mapping $f\colon X\to Y$ we define a variant of $h_f$ by
\begin{equation}\label{eq:hf prime def}
h_f'(x):=\liminf_{r\to 0,\,d(y,x)<r}\frac{L_f(y,5r)}{l_f(y,r)},\quad x\in X.
\end{equation}

Given a BV function $f=(f_1,f_2)\in\BV(\R^2;\R^2)$, for every Lebesgue point $x$ of
$\tfrac{dDf}{d|Df|}$ with respect
to $|Df|$ we have that $\tfrac{dDf}{d|Df|}$ is close to a constant in a small neighborhood of $x$.
Intuitively, for this reason it is sufficient to consider the following type of $f$.
Let $C_1,C_2$ be closed cones with $C_1\cap C_2=\{0\}$ and $-C_1\cap C_2=\{0\}$,
and suppose $f'_1,f'_2\in\BV(\Om)$ with $\tfrac{dDf'_{1}}{d|Df'_1|}(x)\in C_1$
for $|Df'_1|$-a.e. $x\in \Om$
and $\tfrac{dDf'_{2}}{d|Df'_2|}(x)\in C_2$ for $|Df'_2|$-a.e. $x\in \Om$.
Then let $f_1(y):=f'_1(y)+v_1\cdot y$, $f_2(y):=f'_2(y)+v_2\cdot y$, and $f=(f_1,f_2)$.
Denote by $f^*$ the so-called precise representative of $f$; definitions will
be given in Section \ref{sec:prelis}.

\begin{theorem}\label{thm:main Euclidean}
	Let $\Om\subset \R^2$ be open and convex, and suppose
	$f\in \BV(\Om;\R^2)$ is as above. Then $f^*$ is injective, and for $|Df|$-a.e. $x\in \Om$,
	if $h'_{f^*}(x)=\infty$ then $\tfrac{dDf}{d|Df|}(x)$ has rank one.
\end{theorem}

This theorem says that if $f$ is ``non-quasiconformal'' at a point, then
$\tfrac{dDf}{d|Df|}(x)$ has rank one. Thus the problem is reduced to proving such
``non-quasiconformal'' behavior. It turns out that this can be shown in very general metric measure
spaces.

\begin{theorem}\label{thm:main metric space}
	Suppose $(X,d,\mu)$ and $(Y,d_Y,\nu)$ are
	Ahlfors $Q$-regular spaces with $1<Q<\infty$, and $X$ is locally compact.
	Let $\Om\subset X$ be open, and suppose
	$f\in \BV(\Om;Y)$ is bounded and injective.
	Then $h_{f}'(x)=\infty$ for $\Vert Df\Vert^s$-a.e. $x\in \Om$.
\end{theorem}

This theorem demonstrates that the essence of the rank one theorem
is manifest also in metric measure spaces.
We do not know whether the injectivity assumption is necessary here.

As a corollary of these two theorems combined with the well-known reductions,
we obtain the rank one theorem.

\begin{corollary}\label{cor:Alberti}
Let $f\in \BV(\R^n;\R^k)$ with $n,k\in \N$.
Then $\tfrac{dDf}{d|Df|}(x)$ has rank one for $|D^s f|$-a.e. $x\in \R^n$.
\end{corollary}

Quasiconformal mappings have $W_{\loc}^{1,n}(\R^n;\R^n)$-regularity;
Theorem \ref{thm:main metric space} is in the spirit of an extensive literature in the
theory of quasiconformal mappings,
whose message is that $h_f$ being sufficiently small implies that $f$ has at least some
Sobolev regularity.
See e.g. Gehring \cite{Ge62,Ge}, 
Balogh--Koskela--Rogovin \cite{BKR},
Fang \cite{Fan},
Kallunki--Koskela \cite{KaKo},
Kallunki--Martio \cite{KaMa},
Margulis--Mostow \cite{MaMo},
and Williams \cite{Wi}.
As we will discuss in some more detail at the end of the paper,
heuristically the rank one theorem thus turns out to be a special case of this quasiconformal theory.

\section{Notation and definitions}\label{sec:prelis}

In this section we introduce the notation, definitions,
and assumptions that are employed in the paper.

Throughout the paper, we consider two separable  metric measure spaces
$(X,d,\mu)$ and $(Y,d_Y,\nu)$, where $\mu$ and $\nu$ are Borel regular outer measures.
We assume $X$ to be also locally compact, and we assume that both $X$ and $Y$ consist of at least $2$ points,
that is, $\diam X>0$ and $\diam Y>0$.
In Section \ref{sec:main Euc}, we will work with $X=Y=\R^2$ equipped with the Euclidean distance
and the Lebesgue measure $\mathcal L^2$.
For $x\in X$ and $r>0$, an open ball is $B(x,r):=\{y\in X\colon d(y,x)<r\}$ 
and a closed ball is $\overline{B}(x,r):=\{y\in X\colon d(y,x)\le r\}$.
Usually we work with open balls.
We assume every ball to have nonzero and finite measure, in both spaces.
For a ball $B=B(x,r)$, we sometimes denote $2B:=B(x,2r)$.
Note that in a metric space, a ball (as a set) might not have a unique center and radius, but
when using this abbreviation we will work with balls for which these have been specified.
If a property holds outside a set of zero $\mu$-measure, we say that it holds $\mu$-a.e.

Given $1<Q<\infty$, we say that $(X,d,\mu)$ is Ahlfors $Q$-regular if
there is a constant $1\le C_A<\infty$ such that
\[
C_A^{-1}r^Q\le \mu(B(x,r))\le C_A r^Q
\]
for all $x\in X$ and $0<r<2\diam X$.

A continuous mapping from a compact interval into $X$ is said to be a rectifiable curve if it has finite length.
A rectifiable curve $\gamma$ always admits an arc-length parametrization,
so that we get a curve $\gamma\colon [0,\ell_{\gamma}]\to X$
(for a proof, see e.g. \cite[Theorem 3.2]{Haj}).
We will only consider curves that are rectifiable and
arc-length parametrized.
If $\gamma\colon [0,\ell_{\gamma}]\to X$ is a curve and
$g\colon X\to [0, \infty]$ is a Borel function, we define
\[
\int_{\gamma} g\,ds:=\int_0^{\ell_{\gamma}}g(\gamma(s))\,ds.
\]

Let $\Om\subset X$ be open. A mapping $f\colon \Omega\to Y$ is said to 
be $\mu$-measurable if $f^{-1}(W)$ is a $\mu$-measurable set for every open set $W\subset Y$.
By the Kuratowski embedding theorem, see e.g. \cite[p. 100]{HKSTbook},
the metric space $Y$ can be isometrically embedded into a Banach space $(V,\Vert \cdot\Vert_{V})$.
We understand such an embedding to be fixed.
We say that $f\in L^1(\Omega;Y)$ if $f$ is $\mu$-measurable and $\Vert f(\cdot)\Vert_{V}\in L^1(\Omega)$.

To define the class of BV mappings,
we consider the following definitions from Martio \cite{Mar}.
Given a family of curves $\Gamma$, we say that a sequence of nonnegative Borel functions $\{\rho_i\}_{i=1}^{\infty}$
is $\AM$-admissible for $\Gamma$ if
\[
\liminf_{i\to\infty}\int_{\gamma}\rho_i\,ds\ge 1\quad \textrm{for all }\gamma\in\Gamma.
\]
Then we let
\[
\AM(\Gamma):=\inf \left\{\liminf_{i\to\infty}\int_X \rho_i\,d\mu\right\},
\]
where the infimum is taken over all $\AM$-admissible sequences $\{\rho_i\}_{i=1}^{\infty}$.

\begin{definition}\label{def:BV def}
	Let $\Om\subset X$ be open.
	Given a mapping $f\colon \Om\to Y$, we say that $f$ is in the Dirichlet class
	$D^{\BV}(\Om;Y)$ if there exists a sequence of
	nonnegative functions $\{g_i\}_{i=1}^{\infty}$ that is bounded in $L^1(\Om)$, such that
	for $\AM$-a.e. curve $\gamma\colon [0,\ell_{\gamma}]\to \Om$, we have
	\[
		d_Y(f(\gamma(t_1)),f(\gamma(t_2)))
		\le \liminf_{i\to\infty}\int_{\gamma|_{[t_1,t_2]}}g_i\,ds
	\]
	for almost every $t_1,t_2\in [0,\ell_{\gamma}]$ with $t_1<t_2$.
	We also define the total variation
	\[
	\Vert Df\Vert(\Om):=\inf \left\{\liminf_{i\to\infty}\int_{\Om}g_i\,d\mu\right\},
	\]
	where the infimum is taken over sequences $\{g_i\}_{i=1}^{\infty}$ as above.
	If also $f\in L^1(\Om;Y)$, then we say that $f\in \BV(\Om;Y)$.
\end{definition}

For an arbitrary set $A\subset X$, we define
\[
\Vert Df\Vert (A):=\inf\{\Vert Df\Vert (W)\colon A\subset W,\,W\subset X
\text{ is open}\}.
\]
If $f\colon \Om\to Y$ and $\Vert Df\Vert(\Omega)<\infty$,
then $\Vert Df\Vert$ is
a Borel regular outer measure on $\Omega$, see \cite[Theorem 4.1]{Mar}.
In \cite{Mar}, real-valued functions were considered in place of
$Y$-valued mappings, but the same proofs apply almost verbatim.
See also \cite[Theorem 3.4]{Mir} for a proof of the measure property for Banach space valued BV
mappings, with a slightly different definition of the total variation. 
When $(X,d,\mu)$ is an Ahlfors $Q$-regular space,  the variation measure of $f\in \BV(\Om;Y)$
can be decomposed into the absolutely continuous and singular parts with respect to $\mu$
as $\Vert Df\Vert=\Vert Df\Vert^a+\Vert Df\Vert^s$, see e.g. \cite[p. 82]{HKSTbook}.

Next, we consider the Euclidean theory.
We follow mostly the monograph Ambrosio--Fusco--Pallara \cite{AFP}.
The $n$-dimensional Lebesgue measure is denoted by $\mathcal L^n$, with $n\ge 1$.
The $s$-dimensional Hausdorff content is denoted by $\mathcal H^s_R$, $0< R\le \infty$,
and the corresponding Hausdorff measure by $\mathcal H^s$.
We always work with the Euclidean norm $|\cdot|$ for vectors $v\in \R^n$ as well as for
matrices $A\in \R^{k\times n}$.
The inner product of $x,y\in\R^n$ is denoted by $x\cdot y$.

For an open set $\Om\subset \R^n$, we denote by
$C_c(\Om;\R^l)$ the space of continuous $\R^l$-valued functions with compact support in
$\Om$, $l\in\N$.
We denote by $\mathcal M(\Om;\R^{l})$ the Banach space of vector-valued finite Radon measures.
By finite we mean that the total variation $|\nu|(\Om)$ of $\nu\in \mathcal M(\Om;\R^{l})$ is finite.
We denote the set of positive Radon measures by $\mathcal M^+(\Om)$.
For a vector-valued Radon measure $\nu\in\mathcal M(\Om;\R^l)$ and a positive Radon
measure $\mu\in\mathcal M^+(\Om)$, we can write the Radon-Nikodym decomposition $\nu=\nu^a+\nu^s=\frac{d\nu}{d\mu}\mu+\nu^s$ of $\nu$ with respect to
$\mu$, where $\frac{d\nu}{d\mu}\in L^1(\Om,\mu;\R^l)$.
For $\nu_j,\nu \in \mathcal M(\Om;\R^l)$,
weak* convergence $\nu_j\overset{*}{\rightharpoondown}\nu$ in $\Om$ means
\[
\int_{\Om} \phi \cdot \frac{d\nu_j}{d|\nu_j|}\,d|\nu_j|
\to \int_{\Om} \phi \cdot \frac{d\nu}{d|\nu|}\,d|\nu|
\]
for all $\phi\in C_c(\Om;\R^l)$.

The following theory of BV functions is from \cite[Section 3]{AFP}.
Let $k\in\N$ and let $\Om\subset\R^n$ be an open set. 
A function
$f\in L^1(\Omega;\R^k)$ is a function of bounded variation,
denoted $f\in \BV(\Omega;\R^k)$, if its weak derivative
is an $\R^{k\times n}$-valued Radon measure with finite total variation. This means that
there exists a (unique) $Df\in\mathcal M(\Om;\R^{k\times n})$
such that for all $\varphi\in C_c^1(\Omega)$, the integration-by-parts formula
\[
\int_{\Omega}f_j\frac{\partial\varphi}{\partial y_l}\,d\mathcal L^n
=-\int_{\Omega}\varphi\,d(Df_j)_l,\quad j=1,\ldots,k,\ l=1,\ldots,n,
\]
holds.
If we do not know a priori that a function $f\in L^1_{\loc}(\Om;\R^k)$
is a BV function, then we consider
\begin{equation}\label{eq:definition of total variation}
	\Var(f,\Om):=\sup\left\{\sum_{j=1}^{k}\int_{\Om}f_j\dive\varphi_j\,d\mathcal L^n,\,\varphi\in C_c^{1}(\Om;\R^{k\times n}),
	\,|\varphi|\le 1\right\}.
\end{equation}
If $\Var(f,\Om)<\infty$, then the Radon measure $Df$
exists and $\Var(f,\Om)=|Df|(\Om)$
by the Riesz representation theorem, and $f\in\BV(\Om;\R^k)$ provided that $f\in L^1(\Om;\R^k)$.
In the case $k=1$, we denote $\BV(\Om)=\BV(\Om;\R)$.

We have now defined two versions of the total variation. The following proposition shows that they
are comparable in Euclidean spaces.

\begin{proposition}\label{prop:two variation measures}
	Let $\Om\subset \R^n$ be an open set
	and let $f\in L^1_{\loc}(\Om;\R^k)$, with $k\in\N$. Then
	\[
	C^{-1}\Vert Df\Vert(\Om)\le \Var(f,\Om)\le C\Vert Df\Vert(\Om)
	\]
	for a constant $1\le C<\infty$ depending only on $n,k$.
\end{proposition}
\begin{proof}
	Writing $f=(f_1,\ldots,f_k)$, for all $j=1,\ldots,k$ we clearly have
	\[
	\Var(f_j,\Om)\le \Var(f,\Om)\le \Var(f_1,\Om)+\ldots+\Var(f_k,\Om).
	\]
	The analogous property is straightforward to show also for $\Vert Df\Vert(\Om)$.
	Thus it is enough to consider the case $k=1$.
	Using Theorem 3.10 of Durand-Cartagena--Eriksson-Bique--Korte--Shanmugalingam \cite{DCEBKS}
	and localizing the proof of \cite[Proposition 3.3]{DCEBKS},
	we obtain the comparability of $\Vert Df\Vert(\Om)$
	with Miranda's \cite[Definition 3.1]{Mir} definition of total variation.
	On the other hand, the latter agrees with $\Var(f,\Om)$, as noted in
	\cite[Proposition 1.1]{Mir}.
\end{proof}

We denote the characteristic function of a set $E\subset\R^n$ by $\ch_E\colon \R^n\to \{0,1\}$.
If $E\subset\R^n$ with $\Var(\ch_E,\R^n)<\infty$, we say that $E$ is a set of finite perimeter.
We have the following relative isoperimetric inequality on the plane:
for a  set of finite perimeter $E\subset \R^2$, $x\in \R^2$, and $r>0$, we have
\begin{equation}\label{eq:rel isop ineq}
	\min\{\mathcal L^2(B(x,r)\cap E),\mathcal L^2(B(x,r)\setminus E)\}
	\le r |D\ch_E|(B(x,r)).
\end{equation}
The coarea formula states that for an open set $\Om\subset \R^n$ and
a function $u\in \BV(\Om)$, we have
\begin{equation}\label{eq:coarea}
	|Du|(\Om)=\int_{-\infty}^{\infty}|D\ch_{\{u>t\}}|(\Om)\,dt.
\end{equation}
Here we abbreviate $\{u>t\}:=\{x\in \Om\colon u(x)>t\}$.

Given a $\mathcal L^n$-measurable set $D\subset \R^n$ with nonzero and finite
Lebesgue measure, and a function $f\in L^1(D)$, we denote
\[
f_D:=\vint{D}f\,d\mathcal L^n:=\frac{1}{\mathcal L^n(D)}\int_D f\,d\mathcal L^n.
\] 

Let $f\in L^1_{\loc}(\Om)$.
We will often consider a particular pointwise
representative, namely the precise representative
\begin{equation}\label{eq:precise representative}
	f^*(x):=\inf\left\{t\in\R\colon
	\lim_{r\to 0} \frac{\mathcal L^n(B(x,r)\cap \{f>t\})}{\mathcal L^n(B(x,r))}=0\right\},\quad x\in\Om.
\end{equation}
For $f\in L^1_{\loc}(\Om;\R^k)$, we also define $f^*:=(f_1^*,\ldots,f_k^*)$.
We say that $x\in\Om$ is a Lebesgue point of $f\in L^1_{\loc}(\Om;\R^k)$ if
\[
\lim_{r\to 0}\,\vint{B(x,r)}|f(y)-\widetilde{f}(x)|\,d\mathcal L^n(y)=0
\]
for some $\widetilde{f}(x)\in\R^k$. We denote by $S_f\subset\Om$ the set where
this condition fails and call it the approximate discontinuity set.
Given a unit vector $\nu\in \R^n$,  $x\in\R^n$, and $r>0$, we define the half-balls
\begin{align*}
	B_{\nu}^+(x,r)\coloneqq \{y\in B(x,r)\colon  (y-x)\cdot \nu>0\},\\
	B_{\nu}^-(x,r)\coloneqq \{y\in B(x,r)\colon  (y-x)\cdot \nu<0\}.
\end{align*}
We say that $x\in \Om$ is an approximate jump point of $f$ if there exist a unit vector $\nu\in \R^n$
and distinct vectors $f^+(x), f^-(x)\in\R^k$ such that
\[
	\lim_{r\to 0}\,\vint{B_{\nu}^+(x,r)}|f(y)-f^+(x)|\,d\mathcal L^n(y)=0
\]
and
\[
	\lim_{r\to 0}\,\vint{B_{\nu}^-(x,r)}|f(y)-f^-(x)|\,d\mathcal L^n(y)=0.
\]
The set of all approximate jump points is denoted by $J_f$.
For $f\in\BV(\Om;\R^k)$, we have that $\mathcal H^{n-1}(S_f\setminus J_f)=0$, see \cite[Theorem 3.78]{AFP}.

We write the Radon-Nikodym decomposition of the variation measure of
$f\in\BV(\Om;\R^k)$ into the absolutely continuous and singular parts with respect to $\mathcal L^n$
as $Df=D^a f+D^s f$. Furthermore, we define the Cantor and jump parts of $Df$ by
\[
D^c f\coloneqq  D^s f\mres (\Om\setminus S_f),\qquad D^j f\coloneqq D^s f\mres J_f;
\]
here
\[
D^s f \mres J_f(A):=D^s f (J_f\cap A),\quad \textrm{for } D^s f\textrm{-measurable } A\subset \Om.
\]
Since $\mathcal H^{n-1}(S_f\setminus J_f)=0$ and $|Df|$ vanishes on
$\mathcal H^{n-1}$-negligible sets, we get the decomposition (see \cite[Section 3.9]{AFP})
\[
	Df=D^a f+ D^c f+ D^j f.
\]
For the jump part, there is a unit vector $\nu_f\colon J_f\to \R^n$ such that
(note that $(f^{+}-f^-)	\otimes \nu_f $ is a $k\times n$-matrix)
\begin{equation}\label{eq:jump part representation}
	D^j f=(f^{+}-f^-)	\otimes \nu_f \mathcal H^{n-1}\mres J_f.
\end{equation}

If we have a sequence $\{f_i\}_{i=1}^{\infty}\subset \BV(\Om;\R^k)$ and $f_i\to f$ in $L^1(\Om;\R^k)$
and $|Df_i|(\Om)\to |Df|(\Om)$, then we say that $f_i\to f$ strictly in $\BV(\Om;\R^k)$.

The (Sobolev) $1$-capacity of a set $A\subset \R^n$ is defined by
\[
\capa_1(A):=\inf \Vert u\Vert_{W^{1,1}(\R^n)},
\]
where the infimum is taken over Sobolev functions $u\in W^{1,1}(\R^n)$ satisfying
$u\ge 1$ in a neighborhood of $A$.

Given $v\in\R^2$ with $|v|=1$ and $0<a\le 1$, we define a closed cone by
\[
C(v,a):=\{x\in\R^2\colon \,x\cdot v\ge a|x|\}.
\]

\section{Proof of Theorem \ref{thm:main Euclidean}}\label{sec:main Euc}

In this section we prove Theorem \ref{thm:main Euclidean}.

We start with some preliminary results.
As noted in the introduction, our proof of the rank one theorem  follows partially
along the usual steps and in particular relies on the following two propositions
that essentially reduce the problem to monotone functions $f\colon\R^2\to \R^2$;
see \cite[Proposition 1.3, Proposition 5.1]{DLe}.

\begin{proposition}\label{prop:plane reduction}
	The rank one theorem (Corollary \ref{cor:Alberti}) holds if and only if for
	every $f\in \BV(B(0,1);\R^2)$, $\tfrac{dDf}{d|Df|}(x)$ has rank one
	for $|D^s f|$-a.e. $x\in B(0,1)$.
\end{proposition}

\begin{proposition}\label{prop:standard}
	Let $v\in\R^2$ with $|v|=1$ and $a>a'>0$, suppose $u\in \BV(B(x,r))\cap L^{\infty}(B(x,r))$
	for a ball $B(x,r)\subset \R^2$, and let
	\[
	A:=\left\{y\in B(x,r)\colon\, \frac{dDu}{d|Du|}(y)\in C(v,a)\right\}.
	\]
	Then there exists $w\in \BV(B(x,r))\cap L^{\infty}(B(x,r))$ such that
	$|Du|\mres_{A}\ll |Dw|$ in $B(x,r)$, and
	\[
	\frac{dDw}{d|Dw|}(y)\in C(v,a')\quad\textrm{for }|Dw|\textrm{-a.e. }y\in B(x,r).
	\]
\end{proposition}

\begin{lemma}\label{lem:H and capa}
Let $A\subset \R^2$. Then $\mathcal H^1_{\infty}(A)\le 10\capa_1(A)$.
\end{lemma}
\begin{proof}
	We can assume that $\capa_1(A)<\infty$.
	Let $\eps>0$.
	We find a function $u\in W^{1,1}(\R^2)$ such that
	$u\ge 1$ in a neighborhood of $A$, and
	\[
	\int_{\R^2} |\nabla u|\,d\mathcal L^2\le \capa_1(A)+\eps.
	\]
	Here $u\in W^{1,1}(\R^2)\subset  \BV(\R^2)$ with $|Du|(\R^2)= \int_{\R^2} |\nabla u|\,d\mathcal L^2$,
	and then by the coarea formula \eqref{eq:coarea} we find
	a set $E:=\{u>t\}$ for some $0<t<1$, for which
	\[
		|D\ch_{E}|(\R^2)\le |Du|(\R^2) \le \capa_1(A)+\eps,
	\]
	and  $A$ is contained in the interior of $E$.
Then necessarily $\mathcal L^2(E)<\infty$, and for every $x\in A$ we find $r_x>0$ such that
\[
\frac{\mathcal L^2(B(x,r_x)\cap E)}{\mathcal L^2(B(x,r_x))}=\frac{1}{2}.
\]
From the relative isoperimetric inequality \eqref{eq:rel isop ineq}, we get
\[
\frac{\pi r_x^2}{2}
=\min\{\mathcal L^2(B(x,r_x)\cap E),\mathcal L^2(B(x,r_x)\setminus E)\}
\le r_x|D\ch_{E}|(B(x,r_x)).
\]
In particular, the radii $r_x$ are uniformly bounded from above by $(2/\pi)|D\ch_E|(\R^2)$.
By the $5$-covering theorem (see e.g. \cite[p. 60]{HKSTbook}),
we can choose a finite or countable collection $\{B(x_j,r_j)\}_{j}$ of pairwise disjoint balls such that
the balls $B(x_j,5r_j)$ cover $A$.
Then
\begin{align*}
\mathcal H_{\infty}^1(A)
\le \sum_{j}10r_j
\le \frac{20}{\pi}\sum_{j}|D\ch_{E}|(B(x_j,r_j))
\le 10 |D\ch_{E}|(\R^2)
\le 10(\capa_1(A)+\eps).
\end{align*}
Letting $\eps\to 0$, we get the result.
\end{proof}

For $x\in\R^2$, let $p(x):=|x|$.

\begin{lemma}\label{lem:polar projection}
Let $A\subset\R^2$. Then we have $\mathcal L^1(p(A))\le 10\capa_1(A)$.
\end{lemma}
\begin{proof}
Note that $p$ is a $1$-Lipschitz function. Thus we estimate
\[
\mathcal L^1(p(A))
=\mathcal H^1_{\infty}(p(A))
\le \mathcal H^1_{\infty}(A)
\le 10\capa_1(A)
\]
by Lemma \ref{lem:H and capa}.
\end{proof}

The following is a special case of \cite[Theorem 3.10]{LaSP}.

\begin{theorem}\label{thm:strict and uni conv}
	Let $\Om\subset \R^2$ be open and suppose that $u_i\to u$ strictly in $\BV(\Om)$.
	Then we find a subsequence of $u_i$ (not relabeled)
	such that for every compact set $K\subset \Om\setminus S_u$
	and every $\eps>0$, there is an open set $G\subset \Om$ 
	 such that $\capa_1(G)<\eps$ and $u_i^*\to u^*$ uniformly in $K\setminus G$.
\end{theorem}

\begin{proposition}\label{prop:spehere uniform conv}
Suppose $u_i\to u$ strictly in $\BV(B(0,R))$ with $R>0$, and that $S_u=\emptyset$.
Then we find a subsequence (not relabeled) such that for a.e. $0<r<R$ ,
$u_i^*\to u^*$ uniformly in $\partial B(0,r)$.
\end{proposition}
\begin{proof}
Choose a subsequence (not relabeled) given by Theorem \ref{thm:strict and uni conv}.
We find open sets $G_j\subset B(0,R)$
such that $\capa_1(G_j)<1/j$ and
$u_i^*\to u^*$ uniformly in $\overline{B}(0,R-1/j)\setminus G_j$.
By Lemma \ref{lem:polar projection}, we have
$\mathcal L^1(p(G_j))< 10/j$,
and so  $u_i^*\to u^*$ uniformly in $\partial B(0,r)$ for all $r\in (0,R-1/j)\setminus p(G_j)$,
with $\mathcal L^1((0,R-1/j)\setminus p(G_j))>R-11/j$.
Thus  $u_i^*\to u^*$ uniformly in $\partial B(0,r)$ for a.e. $0<r<R$.
\end{proof}

\begin{lemma}\label{lem:two cones}
Let $v_1,v_2\in\R^2$ with $|v_1|=|v_2|=1$ and let $a_1,a_2>0$ such that
$C_1\cap C_2=\{0\}$ and $-C_1\cap C_2=\{0\}$ for
$C_1:=C(v_1,a_1)$ and $C_2:=C(v_2,a_2)$.
Then there exists $\delta>0$, depending only on the two cones, such that for 
every $V\in\R^2$ we have $|V\cdot z|\ge \delta|V||z|$ either for all $z\in C_1$
or for all $z\in C_2$.
\end{lemma}
\begin{proof}
Supposing the claim is false, then there exist $V\in \R^2$ and
$z_{1,j}\in C_1\setminus \{0\}$ and $z_{2,j}\in C_2\setminus \{0\}$
such that $|V\cdot z_{1,j}|/(|V||z_{1,j}|)\to 0$ and $|V\cdot z_{2,j}|/(|V||z_{2,j}|)\to 0$ as $j\to\infty$.
We can assume that $|V|=|z_{1,j}|=|z_{2,j}|=1$.
Passing to subsequences (not relabeled) we get $z_{1,j}\to z_{1}\in C_1$ and $z_{2,j}\to z_{2}\in C_2$
with $|z_1|=|z_2|=1$ and $|V\cdot z_{1}|= 0=|V\cdot z_{2}|$.
Thus either $z_1=z_2$ or $z_1=-z_2$, which contradicts the assumption
$C_1\cap C_2=\{0\}$ and $-C_1\cap C_2=\{0\}$.
\end{proof}

Given a positive Radon measure $\kappa$ on $\R^2$, for a fixed $x\in\R^2$ we define for $r>0$
\begin{equation}\label{eq:kappa def}
	[\kappa]_r(A):=\frac{\kappa(x+r A)}{\kappa(B(x,r))}, \quad A\subset B(0,1).
\end{equation}

The following is a special case of Larsen \cite[Lemma 5.1]{Lar}.

\begin{lemma}\label{lem:Larsen}
	Let $\kappa$ be a positive Radon measure on an open set $\Om\subset \R^2$.
	Then for $\kappa$-a.e. $x\in \Om$, there is a sequence $r_j\searrow 0$ and a Radon measure
	$\kappa'$ such that
	\[
	[\kappa]_{r_j}\overset{*}{\rightharpoondown} \kappa'\textrm{ in }B(0,1)\  \textrm{ with }
	\kappa'(B(0,1))=1.
	\]
\end{lemma}

In \eqref{eq:hf prime def} we defined $h_f'$; more explicitly, the definition reads
\[
h_f'(x)
:=\inf\left\{\liminf_{j\to \infty} \frac{L_f(y_j,5r_j)}{l_f(y_j,r_j)}\colon r_j\to 0,\,y_j\to x,\, 
|y_j-x|< r_j\right\}.
\]
In addition to the Euclidean norm $|A|$ for a matrix $A\in \R^{k\times n}$, we also consider
the maximum norm
\[
\Vert A\Vert_{\textrm{max}}:=\max_{v\in \R^n,|v|=1}|Av|.
\]
The following is a restating of Theorem \ref{thm:main Euclidean}.

\begin{theorem}\label{thm:monotone and injective}
	Let $\Om\subset \R^2$ be open and convex.
	Let $v_1,v_2\in\R^2$ with $|v_1|=|v_2|=1$ and let $a_1,a_2>0$ such that
	$C_1\cap C_2=\{0\}$ and $-C_1\cap C_2=\{0\}$ for
	$C_1:=C(v_1,a_1)$ and $C_2:=C(v_2,a_2)$.
	Suppose $f'_1,f'_2\in\BV(\Om)$ with $\tfrac{dDf'_{1}}{d|Df'_1|}(x)\in C_1$
	for $|Df'_1|$-a.e. $x\in \Om$
	and $\tfrac{dDf'_{2}}{d|Df'_2|}(x)\in C_2$ for $|Df'_2|$-a.e. $x\in \Om$.
	Let $f_1(y):=f'_1(y)+v_1\cdot y$, $f_2(y):=f'_2(y)+v_2\cdot y$, and $f=(f_1,f_2)$.
	Then $f^*$ is injective, and for $|Df|$-a.e. $x\in \Om$,
	if $h'_{f^*}(x)=\infty$ then $\tfrac{dDf}{d|Df|}(x)$ has rank one.
\end{theorem}
\begin{proof}
	Denote standard mollifiers by $\phi_{\delta}$, $\delta>0$,
	and consider the mollifications
	\[
	f_{\delta}:=\phi_{\delta}*f=(\phi_{\delta}*f_1,\phi_{\delta}*f_2)=:(f_{1,\delta},f_{2,\delta}).
	\]
	Denote $\Om_{\delta}:=\{y\in \Om\colon \dist(y,\R^2\setminus \Om)>\delta\}$.
	For every $\delta>0$ we have
	$f_{1,\delta},f_{2,\delta}\in C^{\infty}(\Om_{\delta})$ and
	\[
	\nabla f_{1,\delta}
	=\phi_{\delta}*D f_1
	=\phi_{\delta}*D f_1'+v_1.
	\]
	The analog holds for $f_{2,\delta}$.
	Let $x,y\in \Om_{\delta}$ with $x\neq y$.
	Then $|(y-x)\cdot z|>0$ either for all $z\in C_1\setminus \{0\}$
	or for all $z\in C_2\setminus \{0\}$ by Lemma \ref{lem:two cones}.
	We can assume that in fact
	$(y-x)\cdot z>0$ for all $z\in C_1\setminus \{0\}$.
	Now
	\begin{equation}\label{eq:injectivity}
	\begin{split}
	f_{1,\delta}(y)-f_{1,\delta}(x)
	&=\int_0^{1}\nabla f_{1,\delta}(x+t(y-x))\cdot (y-x)\,dt\\
	&=\int_0^{1}\phi_{\delta}*D f_1'(x+t(y-x))\cdot (y-x)+v_1\cdot (y-x)\,dt\\
	&\ge v_1\cdot (y-x)>0.
	\end{split}
	\end{equation}
	We find a sequence $\delta_j\to 0$ such that
	$f_{\delta_j}\to f^*$ a.e. in $\Om$. For a.e. $x,y\in \Om$,
	if $v_1\cdot (y-x)>0$, then by
	\eqref{eq:injectivity} we have
	\begin{equation}\label{eq:v1 estimate}
	f_{1}^*(y)-f_{1}^*(x)\ge v_1\cdot (y-x)>0.
	\end{equation}
	From the definition of the precise representative \eqref{eq:precise representative},
	this is then true for all $x,y\in \Om$ such that $v_1\cdot(y-x)>0$.
	Similarly considering the other cases $v_1\cdot (y-x) <0$, $v_2\cdot (y-x) >0$,
	and $v_2\cdot (y-x) <0$,
	we conclude that $f^*$ is injective.
	
		Let $x\in \Om$ such that $\tfrac{dDf}{d|Df|}(x)$ has rank two.
	Excluding a $|Df|$-negligible set,
	we can assume that for the $2\times 2$-matrix of full rank $L':=\tfrac{dDf}{d|Df|}(x)$, we have
	the Lebesgue point property
	\begin{equation}\label{eq:Lebesgue for Df}
		\lim_{r\to 0}\,\vint{B(x,r)}\left|\frac{dDf}{d|Df|}-L'\right|\,d|Df|=0.
	\end{equation}
	We can also assume that the conclusion of Lemma \ref{lem:Larsen} holds at $x$ with the
	choice $\kappa=|Df|$.
	To complete the proof, we need to show that
	$h_{f^*}'(x)<\infty$.
	
	Consider the blowups
	\[
	f_r(z):=\frac{f(x+rz)-f_{B(x,r)}}{|Df|(B(x,r))/r},\quad z\in B(0,1),
	\]
	for all $r>0$ such that $B(x,r)\subset \Om$.
	Note that $(f_r)^*=(f^*)_r$, and so we can simply write $f_r^*$.
	For such $r$, just as in \eqref{eq:kappa def} we also define
	\[
	[Df]_r(A):=\frac{Df(x+rA)}{|Df|(B(x,r))}, \quad A\subset B(0,1).
	\]
	Then $Df_r=[Df]_r$ and $|Df_r|=[|Df|]_r$.
	Thus $|Df_r|(B(0,1))=1$.
	Then by the Poincar\'e inequality,
	the quantities $\Vert f_r\Vert_{L^1(B(0,1))}$ are also uniformly bounded.
	Thus by BV compactness, see e.g. \cite[Theorem 3.23]{AFP}, passing to a subsequence $r_j\to 0$
	we have $f_{r_j}\to h\in \BV(0,1)$ in $L^1(B(0,1))$ and also $Df_{r_j}\overset{*}{\rightharpoondown}Dh$,
	and by Lemma \ref{lem:Larsen}, passing to a further subsequence (not relabeled) we have
	$|Df_{r_j}|\overset{*}{\rightharpoondown}\mu$ with  $\mu(B(0,1))=1$.
	We have $[Df]_{r_j}-L'[|Df|]_{r_j}\overset{*}{\rightharpoondown} 0$ by \eqref{eq:Lebesgue for Df},
	and so $Df_{r_j}=[Df]_{r_j}\overset{*}{\rightharpoondown} L'\mu$.
	Thus $Dh=L'\mu$ with $|Dh|=|L'|\mu=\mu$, and then it is easy to see that
	$Dh=cL'\,d\mathcal L^2$ for some $c>0$, see e.g. \cite[Lemma 1.4]{DLe}. Thus $h$
	is a linear mapping of full rank, and
	\[
	|Dh|(B(0,1))=|L'|\mu(B(0,1))=1,
	\]
	and so in fact
	$f_{r_j}\to h$ strictly in $\BV(B(0,1);\R^2)$.
	Denoting $f_{r_j}=(f_{1,r_j},f_{2,r_j})$,
	using e.g. \cite[Proposition 3.15]{AFP} we have also that $f_{1,r_j}\to h_1$
	and $f_{2,r_j}\to h_2$ strictly in $\BV(B(0,1))$.
	Moreover, $h(z)=Lz$, with $L:=cL'$.
	Note also that $L=(V_1,V_2)$ where $V_1\in C_1\setminus \{0\}$ and
	$V_2\in C_2\setminus \{0\}$.
	
	By passing to a further subsequence (not relabeled),
	we can assume that $f_{r_j}^*\to h$ a.e. in $B(0,1)$.
	Choose $w\in B(0,1/20)$ with $f_{r_j}^*(w)\to h(w)$.
	We find $1/2<R<19/20$ such that $|Dh|(\partial B(w,R))=0$, and then 
	also $f_{1,r_j}\to h_1$ and $f_{2,r_j}\to h_2$
	strictly in $B(w,R)$ (see \cite[Example 1.63]{AFP}).
	Passing to a further subsequence (not relabeled),
	by Proposition \ref{prop:spehere uniform conv}
	we find $s\in (1/20,1/10)$ such that
	$f_{r_j}^*\to h$ uniformly in
	$\partial B(w,s)$ and in
	$\partial B(w,5s)$.
	Let $\delta>0$ be the number from Lemma \ref{lem:two cones}.
	Choose $\eps>0$ sufficiently small that
	\begin{equation}\label{eq:choice of eps}
	\frac{5\Vert L\Vert_{\textrm{max}}s+2\eps}{\delta\min\{|V_1|,|V_2|\} s-2\eps}
	\le \frac{6\Vert L\Vert_{\textrm{max}}}{\delta\min\{|V_1|,|V_2|\}}
	\end{equation}
	By the uniform convergence, for large $j$ we have
	\begin{equation}\label{eq:uniform conv sphere}
	\sup_{z\in \partial B(w,s)}|f_{r_j}^*(z)-Lz|<\eps,
	\quad
	\sup_{z\in \partial B(w,5s)}|f_{r_j}^*(z)-Lz|<\eps,
	\quad \textrm{and}\quad
	|f_{r_j}^*(w)- Lw|<\eps.
	\end{equation}
	
	Denote $y_j:=x+r_j w$.
	Let $y\in B(y_j,5r_{j}s)$.
	Let $y'_1,y''_1$ be two points on $\partial B(y_j,5r_{j}s)$
	that are in the $v_1$- and $-v_1$-directions from $y$, respectively.
	Then by \eqref{eq:v1 estimate}, we have
	\[
	f_1^*(y''_1) \le f_1^*(y) \le  f_1^*(y'_1).
	\]
	The analogous fact holds for $f_2^*(y)$, with points $y'_2,y''_2$.
	Now
	\begin{align*}
	&|f^*(y)-f^*(y_j)|
	\le |f_1^*(y)-f_1^*(y_j)|+|f_2^*(y)-f_2^*(y_j)|\\
	&\qquad \le |f_1^*(y'_1)-f_1^*(y_j)|+|f_1^*(y''_1)-f_1^*(y_j)|
	+|f_2^*(y'_2)-f_2^*(y_j)|+|f_2^*(y''_2)-f_2^*(y_j)|\\
	&\qquad \le |f^*(y'_1)-f^*(y_j)|+|f^*(y''_1)-f^*(y_j)|
	+|f^*(y'_2)-f^*(y_j)|+|f^*(y''_2)-f^*(y_j)|.
	\end{align*}
	Thus for large $j$,
	\begin{equation}\label{eq:monotonicity 1}
	\begin{split}
	\sup_{y\in \overline{B}(y_j,5r_{j}s)}|f^*(y)-f^*(y_j)|
	&\le 4\sup_{y\in \partial B(y_j,5r_{j}s)}|f^*(y)-f^*(y_j)|\\
	&= 4\frac{|Df|(B(x,r_j))}{r_j}\sup_{z\in \partial B(w,5s)}|f_{r_j}^*(z)-f_{r_j}^*(w)|\\
	&\le  4\frac{|Df|(B(x,r_j))}{r_j}\left(\sup_{z\in \partial B(w,5s)}|Lz-Lw|+2\eps\right)
	\quad\textrm{by }\eqref{eq:uniform conv sphere}\\
	&\le  4\frac{|Df|(B(x,r_j))}{r_j}\left(5\Vert L\Vert_{\textrm{max}}s+2\eps\right).
	\end{split}
	\end{equation}
	
	Then let $y\in \Om\setminus B(y_j,r_{j}s)$.
	By Lemma \ref{lem:two cones}, we can assume that
	$|(y-y_j)\cdot v|\ge \delta|y-y_j||v|$ for all $v\in C_1$.
	We can further assume that $(y-y_j)\cdot v>0$ for all $v\in C_1$.
	Now $(y-y_j)\cdot v_1>0$ and
	$|(y-y_j)\cdot V_1|\ge \delta|y-y_j||V_1|$.
	Let $y'\in \partial B(y_j,r_{j}s)$ on the line intersecting
	$y_j$ and $y$. 
	Then also
	\begin{equation}\label{eq:delta s V}
	\begin{split}
	|((y'-x)/r_j-w)\cdot V_1|
	& =|((y'-x)/r_j-(y_j-x)/r_j)\cdot V_1| \\
	&\ge \delta |(y'-y_j)/r_j||V_1|\\
	&=\delta s|V_1|.
	\end{split}
	\end{equation}
	Thus
	\begin{align*}
		|f^*(y)-f^*(y_j)|
		&\ge |f_1^*(y)-f_1^*(y_j)|\\
		&\ge |f_1^*(y')-f_1^*(y_j)|\quad\textrm{by }\eqref{eq:v1 estimate}\\
		&= \frac{|Df|(B(x,r_j))}{r_j}|f_{1,r_j}^*((y'-x)/r_j)-f_{1,r_j}^*(w)|\\
		&\ge \frac{|Df|(B(x,r_j))}{r_j}\left(
		|V_1\cdot (y'-x)/r_j-V_1\cdot w|-2\eps\right)
		\quad\textrm{by }\eqref{eq:uniform conv sphere}\\
		&\ge \frac{|Df|(B(x,r_j))}{r_j}\left(\delta|V_1| s-2\eps\right)
		\quad\textrm{by }\eqref{eq:delta s V}\\
		&\ge \frac{|Df|(B(x,r_j))}{r_j}\left(\delta\min\{|V_1|,|V_2|\} s-2\eps\right).
	\end{align*}
	Hence
	\begin{equation}\label{eq:monotonicity 2}
	\inf_{y\in \Om\setminus B(y_j,r_{j} s)}|f^*(y)-f^*(y_j)|
	\ge \frac{|Df|(B(x,r_j))}{r_j}\left(\delta\min\{|V_1|,|V_2|\} s-2\eps\right).
	\end{equation}
	Using \eqref{eq:monotonicity 1} and \eqref{eq:monotonicity 2}, we get
	\begin{align*}
		\liminf_{j\to \infty} \frac{L_{f^*}(y_j,5r_js)}{l_{f^*}(y_j,r_js)}
		&\le \liminf_{j\to\infty}\frac{\sup_{y\in \overline{B}(y_j,5r_{j}s)}|f^*(y)-f^*(y_j)|}
		{\inf_{y\in \Om\setminus B(y_j,r_{j} s)}|f^*(y)-f^*(y_j)|}\\
		&\le 4\frac{5\Vert L\Vert_{\textrm{max}}s+2\eps}{\delta\min\{|V_1|,|V_2|\} s-2\eps}\\
		&\le 4\frac{6\Vert L\Vert_{\textrm{max}}}{\delta\min\{|V_1|,|V_2|\}}\quad\textrm{by }\eqref{eq:choice of eps}\\
		&<\infty,
	\end{align*}
	and so $h_{f^*}'(x)<\infty$.
\end{proof}

\section{Proof of Theorem \ref{thm:main metric space}}\label{sec:main metric}

In this section we prove Theorem \ref{thm:main metric space}
and then Corollary \ref{cor:Alberti}.

We start with the following lemma which is a variant of
\cite[Lemma 3.23]{LahGFQ}.
The oscillation of
$f\colon U\to \R$ in a set $U\subset \R$ is defined by
\[
\underset{U}{\osc}\,f:=\sup\{|f(x)-f(y)|,\,x,y\in U\}.
\]

\begin{lemma}\label{lem:line estimate}
	Let  $f\colon [a,b]\to \R$ for some compact interval $[a,b]\subset \R$.
	Suppose that there is a sequence of at most countable unions of sets
	$W_i=\bigcup_{j}U_{i,j}$, $i\in\N$,
	where each $U_{i,j}\subset \R$ is open and bounded, 
	such that $\ch_{W_i}(t)\to 1$ as $i\to\infty$ for all $t\in [a,b]$.
	Let $A\subset [a,b]$ with $\mathcal L^1([a,b]\setminus A)=0$.
	Then for every $t_1,t_2\in A$ with $t_1<t_2$, we have
	\[
	|f(t_1)-f(t_2)|\le \liminf_{i\to\infty}\sum_j\underset{U_{i,j}\cap A}{\osc}\, f.
	\]
\end{lemma}

\begin{proof}
	Let $t_1,t_2\in A$ with $t_1<t_2$.
	First assume that $f$ is bounded.
	Then, without loss of generality, we can also assume that $f(t_1)<f(t_2)$.
	Define
	\[
	h(x):=\sup_{t\in [t_1, x]\cap A}f(t),\quad t_1\le x\le t_2,
	\]
	and let $h(x):=h(t_1)$ for $x\le t_1$ and $h(x):=h(t_2)$ for $x\ge t_2$.
	Now $h$ is an increasing function, and so $h\in\BV_{\loc}(\R)$.
	Consider a bounded nonempty open set $U\subset \R$.
	We can represent $U$ as an at most countable union of disjoint open intervals $U=\bigcup_{l}U_l$
	with $U_l=(a_l,b_l)$.
	Since $h$ is increasing, in each interval we can consider the one-sided limits
	$h(a_l+)$ and $h(b_l-)$.
	Moreover, in each interval $U_l$, define the truncation
	\[
	f_l:=\min\{h(b_l-),\max\{h(a_l+),f\}\}.
	\]
	Since the intervals $(h(a_l+),h(b_l-))$ are disjoint, we have
	\[
	\sum_{l}\osc_{U_l\cap A} f_l\le \osc_{U\cap A} f.
	\]
	Since $h$ is increasing, for every interval $U_l$ we clearly  have
	\[
	|Dh|(U_l)=h(b_l-)-h(a_l+).
	\]
	On the other hand, we also have (in fact equality holds)
	\begin{align*}
		\osc_{U_l\cap A} f_l
		\ge h(b_l-)-h(a_l+),
	\end{align*}
	because $\inf_{U_l\cap A}f\le h(a_l+)$, and
	either $h(b_l-)=h(a_l+)$ or $h(b_l-)=\sup\{f(y)\colon y\in U_l\cap A\}$.
	Then since $|Dh|$ is a Radon measure, we have
	\begin{equation}\label{eq:osc h}
		|Dh|(U)=\sum_{l}|Dh|(U_{l})
		\le \sum_{l}\osc_{U_l\cap A} f_l\le \osc_{U\cap A} f.
	\end{equation}
	Let $H_i:=\bigcap_{k=i}^{\infty}\bigcup_{j=1}^{\infty}U_{k,j}$.
	Then by assumption, $\bigcup_{i=1}^{\infty}H_i=[a,b]$.
	Noting that $f(t_1)=h(t_1)$, and then using basic properties of the Radon measure $|Dh|$, we estimate
	\begin{align*}
		|f(t_2)-f(t_1)|
		&\le h(t_2)-h(t_1)\\
		&=  |Dh|([t_1,t_2])\\
		&\le \liminf_{i\to\infty}|Dh|(H_i\cap [t_1,t_2])\\
		&\le \liminf_{i\to\infty}\sum_j|Dh|(U_{i,j}\cap [t_1,t_2])\\
		&\le \liminf_{i\to\infty}\sum_j\underset{U_{i,j}\cap A}{\osc}\, f\quad
		\textrm{by }\eqref{eq:osc h}.
	\end{align*}
	
	In the general case, we consider the truncations $f_M:=\min\{M,\max\{-M,f\}\}$.
	For sufficiently large $M$, we have
	\[
	|f(a)-f(b)|
	=|f_M(a)-f_M(b)|
	\le \liminf_{i\to\infty}\sum_j\underset{U_{i,j}\cap A}{\osc}\, f_M
	\le \liminf_{i\to\infty}\sum_j\underset{U_{i,j}\cap A}{\osc}\, f.
	\]
\end{proof}

\begin{lemma}\label{lem:line behavior}
	Let  $f\colon [a,b]\to Y$ for some compact interval $[a,b]\subset \R$.
	Suppose that there is a sequence of at most countable unions of sets
	$W_i=\bigcup_{j}U_{i,j}$, $i\in\N$,
	where each $U_{i,j}\subset \R$ is open and bounded, 
	such that $\ch_{W_i}(t)\to 1$ as $i\to\infty$ for all $t\in [a,b]$.
	Let $A\subset [a,b]$ with $\mathcal L^1([a,b]\setminus A)=0$.
	Then for every $t_1,t_2\in A$ with $t_1<t_2$, we have
	\[
	d_Y(f(t_1),f(t_2))\le \liminf_{i\to\infty}\sum_j \diam f(U_{i,j}\cap A).
	\]
\end{lemma}

\begin{proof}
	Recall that we understand $Y$ to be isometrically embedded into the Banach space
	$(V,\Vert \cdot\Vert_{V})$.
	Denote the dual of $V$ by $V^*$.
	We find $v^*\in V^*$ with $\Vert v^*\Vert_{V^*}=1$ and
	$v^*(f(t_1)-f(t_2))=\Vert f(t_1)-f(t_2)\Vert_V$.
	Now we can consider the real-valued function $v^*\circ f$, and we get
	\begin{align*}
		d_Y(f(t_1),f(t_2))
		&=\Vert f(t_1)-f(t_2)\Vert_V\\
		&= v^*(f(t_1)-f(t_2))\\
		&= v^*(f(t_1))-v^*(f(t_2))\\
		&\le \liminf_{i\to\infty}\sum_j
		\underset{U_{i,j}\cap A}{\osc}\,v^*\circ f\quad\textrm{by Lemma }
		\ref{lem:line estimate}\\
		&\le \liminf_{i\to\infty}\sum_j\diam f(U_{i,j}\cap A).
	\end{align*}
\end{proof}

Observe that if $f\colon\Om\to Y$ is injective, and $B(y_1,r_1)\subset \Om$ and
$B(y_2,r_2)\subset \Om$ are disjoint, then
from the definition of $l_f(\cdot,\cdot)$ and from the injectivity of $f$ it follows that
\begin{equation}\label{eq:injectivity for balls}
	B(f(y_1),l_f(y_1,r_1))\cap B(f(y_2),l_f(y_2,r_2)) = \emptyset.
\end{equation}

\begin{proof}[Proof of Theorem \ref{thm:main metric space}]
	Since $f$ is bounded, there exist $y_0\in Y$ and $0<R<2\diam Y$ such that
	$f(\Om)\subset B(y_0,R)$.
There exists a $\mu$-negligible set $H\subset \Om$ such that $\Vert Df\Vert^s(\Om\setminus H)=0$.
Note that $\{h'_{f}<\infty\}=\bigcup_{M=1}^{\infty}\{h'_{f}<M\}$.
Fix $M\in\N$ and $\eps>0$.
Also choose an open set $U$ such that $H\subset U\subset \Om$
and $\mu(U)<\eps^{Q/(Q-1)}$.
For every
$x\in H\cap \{h'_{f}<M\}$
we have that
\[
\liminf_{r\to 0,\,d(y,x)<r}\frac{L_f(y,5r)}
{l_f(y,r)}<M.
\]
Thus we find $0<r_x<\eps$ and $y_x\in B(x,r_x)$ such that
$\Om\setminus B(y_x,r_x)\neq \emptyset$ (we can assume that $\Om$ consists of at least 2 points),
$B(y_x,r_x)\subset U$, $B(y_x,6r_x)\subset \Om$, and
\[
\frac{L_f(y_x,5r_x)}{l_f(y_x,r_x)}<M.
\]
By the $5$-covering theorem, we find finitely or countably many pairwise disjoint balls
$B_{j}=B(y_{j},r_{j})$
such that the balls $B(y_{j},5r_{j})$ cover $H\cap \{h_f'<M\}$, with
\begin{equation}\label{eq:final M estimate}
\frac{L_f(y_j,5r_j)}
{l_f(y_j,r_j)}<M.
\end{equation}
Let
\begin{equation}\label{eq:g def}
g:=\sum_{j}\frac{L_{f}(y_{j},5r_{j})}{r_{j}}\ch_{6B_{j}}.
\end{equation}
Then for a curve $\gamma\colon [0,\ell_{\gamma}]\to \Om$ with $\ell_{\gamma}>\eps$,
we have
\begin{equation}\label{eq:curve integral}
\sum_{j,\,5B_{j}\cap \gamma\neq \emptyset} \diam f (5B_j)
\le  2\sum_{j,\,5B_{j}\cap \gamma\neq \emptyset} L_{f}(y_{j},5r_{j})
\le 2\int_\gamma g\,ds.
\end{equation}
With $1/Q+(Q-1)/Q=1$,
by Young's inequality we have for any $b_1,b_2\ge 0$ that
\begin{equation}\label{eq:generalized Young}
	\begin{split}
		b_1b_2=\eps^{1/Q} b_1 \eps^{-1/Q} b_2
		&\le \frac{1}{Q}\eps b_1^Q + \frac{Q-1}{Q}\eps^{-1/(Q-1)} b_2^{Q/(Q-1)}\\
		&\le \eps b_1^Q + \eps^{-1/(Q-1)} b_2^{Q/(Q-1)}.
	\end{split}
\end{equation}
Denote $l_j:=l_f(y_j,r_j)$.
Ahlfors regularity holds with the exponent $Q$ and with a constant $C_A$ in both $X$ and $Y$.
Using the Ahlfors regularity, we estimate
\begin{align*}
	&\sum_{j}\frac{L_{f}(y_{j},5r_{j})}{r_{j}} \mu(6 B_{j})\\
	& \qquad  \le 6^QC_A\sum_{j}L_{f}(y_{j},5r_{j}) r_{j}^{Q-1}\\
	& \qquad  \le 6^QC_A M\sum_{j}l_{j} r_{j}^{Q-1}\quad\textrm{by }
	\eqref{eq:final M estimate}\\
	& \qquad   \le  6^{Q^2}C_A^Q M^Q\eps
	 \sum_{j}  l_{j}^Q+ \eps^{-1/(Q-1)}
	\sum_{j} r_{j}^{Q}\quad\textrm{by }\eqref{eq:generalized Young}\\
	& \qquad   =6^{Q^2}C_A^{Q+1} M^Q\eps\sum_{j}  \nu(B(f(y_{j}),l_{j}))
	+C_A\eps^{-1/(Q-1)}\sum_{j} \mu(B_j).
\end{align*}
Note that since $f(\Om)\subset B(y_0,R)$ and $\Om\setminus B(y_{j},r_{j})\neq \emptyset$,
we have $l_{j}< 2R$.
Using injectivity and \eqref{eq:injectivity for balls}, we estimate further
\begin{align*}
	\sum_{j}\frac{L_f(y_{j},5r_{j})}{r_{j}} \mu(6 B_{j})
	&\le 6^{Q^2}C_A^{Q+1} M^Q\eps\nu(B(y_0,3R))
	+C_A\eps^{-1/(Q-1)} \mu(U)\\
		&\le 6^{Q^2}C_A^{Q+1} M^Q\eps\nu(B(y_0,3R))
	+C_A \eps.
\end{align*}
We can do this for $\eps=1/i$, $i=1,2,\ldots$.
Instead of e.g. $y_{j}$, we now denote $y_{i,j}$ to signify also the dependence on $i$.
Using the definition \eqref{eq:g def}, we get a sequence of functions $\{g_i\}_{i=1}^{\infty}$.
For every $i\in\N$, we have
\begin{equation}\label{eq:gj estimate with sum}
	\begin{split}
		\int_{\Om}g_i\,d\mu
		&\le \sum_{j}\frac{L_f(y_{i,j},5r_{i,j})}{r_{i,j}} \mu(6 B_{i,j})\\
		& \le 6^{Q^2}C_A^{Q+1} M^Q(1/i) \nu(B(y_0,3R))+C_A/i\\
		& \to 0\quad\textrm{as }i\to\infty.
	\end{split}
\end{equation}
Recall that we are keeping $M\in\N$ fixed.
Denote
\[
H_M:=\Om\cap \bigcap_{i=1}^{\infty}\bigcup_{j}5B_{i,j}.
\]
Note that $H_M$ is a Borel set that contains $H\cap \{h_f'<M\}$.
Fix $\delta>0$.
Using the measure property of $\Vert Df\Vert$, we get the following
(see e.g. \cite[Proposition 3.3.37]{HKSTbook}):
there is an open set $W\supset H_M$ with $\Vert Df\Vert(W\setminus H_M)<\delta$, and
there is a closed set $F\subset H_M$ with $\Vert Df\Vert(H_M\setminus F)<\delta$.
From the definition of BV mappings (Definition \ref{def:BV def}),
there exists a sequence of
nonnegative functions $\{h_i\}_{i=1}^{\infty}$ such that
for every curve $\gamma$ in $W\setminus F$ outside a family $\Gamma$
with $\AM(\Gamma)=0$, we have
\begin{equation}\label{eq:AM ae curve}
	d_Y(f(\gamma(t_1)),f(\gamma(t_2)))
	\le \liminf_{i\to\infty}\int_{\gamma|_{[t_1,t_2]}}h_i\,ds
\end{equation}
for almost every $t_1,t_2\in [0,\ell_{\gamma}]$ with $t_1<t_2$,
and 
\begin{equation}\label{eq:hi choice}
\int_{W\setminus F} h_i\,d\mu
<\Vert Df\Vert(W\setminus F)+\delta<3\delta\quad\textrm{for every }i\in\N.
\end{equation}
We extend each $h_i$ to $W$ by letting $h_i=0$ in $F$.
Then consider a curve $\gamma\colon [0,\ell_{\gamma}]\to W$; excluding an $\AM$-negligible
family we can assume that
no subcurve of $\gamma$ is in $\Gamma$, as follows readily from the definition of $\AM$-modulus.
Note that $(0,\ell_{\gamma})\setminus \gamma^{-1}(F)$
is an open set, and so it is the
union of pairwise disjoint open intervals $\bigcup_{m=1}^{\infty}V_m$.
By \eqref{eq:AM ae curve},
there exists $A\subset [0,\ell_{\gamma}]$ with $\mathcal L^1([a,b]\setminus A)=0$
and for every $t_1,t_2\in V_m\cap A$ with $t_1<t_2$, we have
\begin{align*}
	\diam f\circ\gamma(V_{m}\cap A)
	&=\sup_{t_1,t_2\in V_{m}\cap A,\,t_1<t_2}d_Y(f(\gamma(t_1)),f(\gamma(t_2)))\\
	&\le \sup_{t_1,t_2\in V_{m}\cap A,\,t_1<t_2}\liminf_{i\to\infty}\int_{\gamma|_{[t_1,t_2]}}h_i\,ds\\
	&\le \liminf_{i\to\infty}\int_{V_m}h_i(\gamma(s))\,ds.
\end{align*}
Let $\delta_{\gamma}>0$.
If
\begin{equation}\label{eq:V im intervals}
\diam f\circ\gamma(V_{m}\cap A)\le \int_{V_{m}}h_i(\gamma(s))\,ds+2^{-m}\delta_{\gamma},
\end{equation}
then let $V_{i,m}=V_m$, and otherwise let $V_{i,m}=\emptyset$.
Note that for every $m$, we have $V_{i,m}=V_m$ for all sufficiently large $i$.
Note also that $\gamma^{-1}(F)$ is covered, for every $i\in\N$, by the open
sets $\gamma^{-1}(5B_{i,j})$.
Now by Lemma \ref{lem:line behavior}, for every $t_1,t_2\in A$ with $t_1<t_2$ we get 
\begin{align*}
	&d_Y(f(\gamma(t_1)),f(\gamma(t_2)))\\
	&\qquad \le \liminf_{i\to\infty}\left(\sum_{j,\,5B_{i,j}\cap\gamma\neq \emptyset}\diam f\circ\gamma(\gamma^{-1}(5B_{i,j}))
	+\sum_m\diam f\circ\gamma(V_{i,m}\cap A)\right)\\
	&\qquad \le \liminf_{i\to\infty}\left(\sum_{j,\,5B_{i,j}\cap\gamma\neq \emptyset}\diam f(5B_{i,j})
	+\sum_m\diam f\circ\gamma(V_{i,m}\cap A)\right)\\
	&\qquad \le \liminf_{i\to\infty}\left(2\int_\gamma g_i\,ds+\int_{\gamma}h_i\,ds+\delta_{\gamma}\right)
	\quad\textrm{by }\eqref{eq:curve integral},\eqref{eq:V im intervals}.
\end{align*}
Since $\delta_{\gamma}>0$ was arbitrary, we get
\[
d_Y(f(\gamma(t_1)),f(\gamma(t_2)))
\le \liminf_{i\to\infty}\int_\gamma (2g_i+h_i)\,ds
\]
for every $t_1,t_2\in A$ with $t_1<t_2$.
Now by definition,
\begin{align*}
\Vert Df\Vert(W)
\le \liminf_{i\to\infty}\int_{W}(2g_i+h_i)\,d\mu
\le 3\delta\quad\textrm{by }\eqref{eq:gj estimate with sum},\eqref{eq:hi choice}.
\end{align*}
Thus $\Vert Df\Vert(H_M)\le \Vert Df\Vert(W)\le 3\delta$.
Since $\delta>0$ was arbitrary, we get $\Vert Df\Vert(H_M)=0$
and so $\Vert Df\Vert(H\cap \{h_f'<M\})=0$, and so $\Vert Df\Vert(H\cap \{h_f'<\infty\})=0$.
Thus
\[
\Vert Df\Vert^s(\{h_f'<\infty\})
=\Vert Df\Vert^s(\Om\setminus H)+\Vert Df\Vert^s(H\cap \{h_f'<\infty\})=0,
\]
proving the result.
\end{proof}

\begin{remark}
Note that in Theorem \ref{thm:main metric space}, we could in fact only assume
$f\in D^{\BV}(\Om;Y)$ instead of $f\in \BV(\Om;Y)$, that is, we do not use the assumption
$f\in L^1(\Om;Y)$.
\end{remark}

\begin{proof}[Proof of Corollary \ref{cor:Alberti}]
By Proposition \ref{prop:plane reduction}, it is enough to consider
$f\in\BV(B(0,1);\R^2)$.
By considering truncations, we can also assume that $f\in L^{\infty}(B(0,1);\R^2)$.
There exists a $\mathcal L^2$-negligible set $H\subset B(0,1)$ such that
$|D^s f|(B(0,1)\setminus H)=0$.
Let $D$ be the subset of $H$ where $\tfrac{dDf}{d|Df|}$ has rank $2$.
Let $\eps>0$. For $|Df|$-a.e. $x\in D$, we find cones
$C_1=C(v_1,a_1)$ and $C_2=C(v_2,a_2)$ with
$C_1\cap C_2=\{0\}$ and $-C_1\cap C_2=\{0\}$, and an arbitrarily small radius $r_x>0$ such that
$|Df|(\partial B(x,r_x))=0$, and
\[
|Df|\left(B(x,r_x)\cap D\cap \left\{\frac{dDf_1}{d|Df|}\in C_1\setminus \{0\}\right\}\right)
>(1-\eps/2)|Df|(B(x,r_x))
\]
and similarly for $f_2$.
Note that if $\tfrac{dDf_1}{d|Df|}\in C_1\setminus \{0\}$, then also $\tfrac{dDf_1}{d|Df_1|}\in C_1\setminus \{0\}$.
By the Vitali covering theorem, see e.g. \cite[Theorem 2.19]{AFP},
we can choose  pairwise disjoint balls $B(x_j,r_j)$ covering $|Df|$-almost all of $D$.
Now for each $j\in\N$, we include the index $j$ in the notation, so that there are
disjoint cones $C_{1,j}=C(v_{1,j},a_{1,j})$ and $C_{2,j}=C(v_{2,j},a_{2,j})$.
Let
\[
A_j:=B(x_j,r_j)\cap D\cap \left\{\frac{dDf_{1}}{d|Df|}\in C_{1,j}\setminus \{0\}\right\}
\cap \left\{\frac{dDf_{2}}{d|Df|}\in C_{2,j}\setminus \{0\}\right\}.
\]
Then $|Df|(A_j)\ge (1-\eps)|Df|(B(x_j,r_j))$.

Let $a'_{1,j}<a_{1,j}$ and $a'_{2,j}<a_{2,j}$ such that for
$C_{1,j}':=C(v_{1,j},a'_{1,j})$ and $C_{2,j}'=C(v_{2,j},a_{2,j}')$, we still have
$C_{1,j}'\cap C_{2,j}'=\{0\}$ and $-C_{1,j}'\cap C_{2,j}'=\{0\}$.
Thus by Proposition \ref{prop:standard}, we find
$f'_{1,j},f'_{2,j}\in \BV(B(x_j,r_j))\cap L^{\infty}(B(x_j,r_j))$ such that
\[
|Df|\mres_{A_j}\ll |Df_1|\mres_{A_j}\ll |Df'_{1,j}|\mres_{A_j}\ll |D^s f'_{1,j}|
\]
and similarly 
$|Df|\mres_{A_j}\ll |D^s f'_{2,j}|$ in $B(x_j,r_j)$, and
\[
\frac{dDf_{1,j}'}{d|Df_{1,j}'|}(x)\in C_{1,j}'
 \textrm{ for }|Df_{1,j}'|\textrm{-a.e. }x\in B(x_j,r_j),
\]
and analogously for $f_{2,j}'$.
Let $f_{1,j}(y):=f'_{1,j}(y)+v_{1,j}\cdot y$, $f_{2,j}(y):=f'_{2,j}(y)+v_{2,j}\cdot y$, and $f_j=(f_{1,j},f_{2,j})$.
Then by Theorem \ref{thm:main Euclidean}, $f^*_j$ is injective.
Note that still $|Df|\mres_{A_j}\ll |D^s f_{1,j}|$ and $|Df|\mres_{A_j}\ll |D^s f_{2,j}|$.
Let
\[
d\kappa:=\min\left\{\frac{d|D^s f_{1,j}|}{d|D^s f_{j}|},\frac{d|D^sf_{2,j}|}{d|D^sf_{j}|}\right\}\,d|D^sf_j|.
\]
Note that $|Df|\mres_{A_j}\ll \kappa$.
Now for $\kappa$-a.e. $x\in B(x_j,r_j)$, we have $\tfrac{d|D^sf_{1,j}|}{d|D^sf_{j}|}(x)>0$ and $\tfrac{d|D^sf_{2,j}|}{d|D^sf_{j}|}(x)>0$, and so
$\tfrac{dDf_j}{d|D f_j|}(x)$ has rank two for $\kappa$-a.e. $x\in B(x_j,r_j)$.
For each $j$, by Theorem \ref{thm:main Euclidean} and Theorem \ref{thm:main metric space}
we have for $|D^s f_j|$-a.e. $x\in B(x_j,r_j)$
that $\tfrac{dDf_j}{d|Df_j|}(x)$ has rank one;
this is then true also for $\kappa$-a.e. $x\in B(x_j,r_j)$.
Here we also rely on Proposition \ref{prop:two variation measures}.
Thus $\kappa(B(x_j,r_j))=0$,
and so $|D^s f|(B(x_j,r_j)\cap A_j)=0$.
In total,
\begin{align*}
|Df|(D)
=|D^sf|(D)
\le \sum_j | D^sf|(B(x_j,r_j)\setminus A_j)
\le \eps\sum_j | Df|(B(x_j,r_j))
\le \eps| Df|(\R^n).
\end{align*}
Letting $\eps\to 0$, we get $|Df|(D)=0$.
\end{proof}

\begin{remark}
In Theorem \ref{thm:main Euclidean} we demonstrate a connection between $\tfrac{dDf}{d|Df|}$
having rank one and $h_{f^*}'$ being infinite.
A similar connection was proved in \cite[Theorem 6.3]{LahGFQ}, where it was shown for
$f\in\BV(\R^n;\R^n)$ that at $|Df|$-a.e. point
$\tfrac{dDf}{d|Df|}$ has full rank
if and only if a slightly different variant of 
$h_{f^*}$ is finite. It should be noted that the proof of this result relied on
Alberti's rank one theorem.
In any case, at least heuristically there is a connection between $\tfrac{dDf}{d|Df|}$ having full rank
and $f$ being ``quasiconformal'' at a point.

If $f\colon \R^n\to \R^n$ is a homeomorphism and $h_f(x)\le H<\infty$ for every $x\in\R^n$,
then $f$ is a quasiconformal mapping; this can be taken as a definition of
quasiconformal mappings. Quasiconformal mapping have, in particular,
$W^{1,n}_{\loc}$-regularity.
There has been wide interest in showing that if this pointwise condition is
 assumed in some relaxed sense,
it follows that the mapping $f$ has at least
 some lower regularity.
For $1\le p\le n$, and denoting $p^*=np/(n-p)$ for $1\le p<n$ and $p^*=\infty$ for $p=n$, 
if $f\colon \R^n\to\R^n$ is a homeomorphism,
$h_f\in L^{p^*}_{\loc}(\R^n)$, and $h_f<\infty$ outside a set of
$\sigma$-finite $\mathcal H^{n-1}$-measure, then $f\in W_{\loc}^{1,p}(\R^n;\R^n)$.
For this, see  Koskela--Rogovin \cite[Corollary 1.3]{KoRo}
and Kallunki--Martio \cite[Theorem 2.2]{KaMa}.

So at one extreme, $h_f$ being uniformly bounded means quasiconformality and thus
$W_{\loc}^{1,n}$-regularity; at the other extreme,
$h_f$ or a variant of it being merely finite in a set is enough to guarantee at least that
the singular part of $\Vert Df\Vert$ does not charge this set, as given by 
Theorem \ref{thm:main metric space}.
Or, $f$ is ``non-quasiconformal'' at $\Vert Df\Vert^s$-a.e. point.
So heuristically, the rank one theorem is one extreme end point of this
wider quasiconformal theory.
\end{remark}

\end{document}